\documentclass[11pt,a4paper]{article}
\usepackage{amsmath,amsfonts,amssymb,latexsym,graphics,epsfig,url}
\usepackage{hyperref}
\usepackage{color}
\usepackage{amsthm}

\usepackage[english]{babel}
\usepackage{epsfig}
\usepackage{graphicx}

\newcommand{\executeiffilenewer}[3]{%
\ifnum\pdfstrcmp{\pdffilemoddate{#1}}%
{\pdffilemoddate{#2}}>0%
{\immediate\write18{#3}}\fi%
}

\newtheorem{theo}{Theorem}[section]
\newtheorem{lema}[theo]{Lemma}
\newtheorem{corollari}[theo]{Corollary}
\newtheorem{propo}[theo]{Proposition}

\DeclareMathOperator{\dist}{dist}

\DeclareMathOperator{\tr}{tr}
\DeclareMathOperator{\som}{sum}
\DeclareMathOperator{\spec}{sp}
\def\vec0{\mbox{\boldmath $0$}}

\def\A{\mbox{\boldmath $A$}}

\def\I{\mbox{\boldmath $I$}}
\def\J{\mbox{\boldmath $J$}}
\def\M{\mbox{\boldmath $M$}}
\def\N{\mbox{\boldmath $N$}}

\def\R{\mbox{\boldmath $R$}}

\def\Re{\mathbb R}

\begin{document}
\title{Some spectral and quasi-spectral \\
characterizations of distance-regular graphs
}

\author{A. Abiad$^{a,b}$, E.R. van Dam$^a$, M.A. Fiol$^c$
\\ \\
{\small $^a$Tilburg University, Dept. of Econometrics and O.R.} \\
{\small  Tilburg, The Netherlands}\\
{\small {\tt Edwin.vanDam@uvt.nl}} \\
{\small $^b$Maastricht University, Dept. of Quantitative Economics} \\
{\small  Maastricht, The Netherlands}\\
{\small {\tt A.AbiadMonge@maastrichtuniversity.nl}} \\
{\small $^c$Universitat Polit\`ecnica de Catalunya, Dept. de Matem\`atiques} \\
{\small Barcelona Graduate School of Mathematics, Barcelona, Catalonia}\\
{\small {\tt fiol@ma4.upc.edu}} \\
 }
\date{}

\maketitle

\begin{abstract}
In this paper we consider the concept of preintersection numbers of a graph. These numbers are determined by the spectrum of the adjacency matrix of the graph, and generalize the intersection numbers of a distance-regular graph. By using the preintersection numbers we give some new spectral and quasi-spectral characterizations of distance-regularity, in particular for graphs with large girth or large odd-girth.
\end{abstract}

{\small \noindent{\em Mathematics Subject Classifications:} 05E30, 05C50.

}

\section{Introduction}
\label{intro}
A central issue in spectral graph theory is to question whether or not a graph is uniquely determined by its spectrum, see the surveys of Van Dam and Haemers \cite{vdh03,vdh09}. In particular, much
attention has been paid to give spectral and quasi-spectral characterizations of distance-regularity.
Contributions in this area are due to Brouwer and Haemers \cite{bh93}, Van Dam and  Haemers \cite{vdh02},
Van Dam, Haemers, Koolen, and Spence \cite{vdhks06},  Haemers \cite{h96}, and Huang and Liu \cite{hl99},
among others.

In this paper, we will give new spectral and quasi-spectral characterizations of distance-regularity of a graph $G$ without requiring, as it is common in this area of research, that:
\begin{itemize}
\item
$G$ is cospectral with a distance-regular graph $\Gamma$, where
\item
$\Gamma$ has intersection numbers, or other combinatorial parameters that satisfy certain properties.
\end{itemize}

The following theorem, given in the recent survey by Van Dam, Koolen, and Tanaka \cite{vdkt12} (see also Van Dam and Haemers \cite{vdh03} and Brouwer and Haemers \cite{bh12}), contains most of the known characterizations of this type.

\begin{theo}
\label{theo-book} If $\Gamma$ is a distance-regular graph with
diameter $D=d$ and girth $g$ satisfying one of the
properties $(i)$--$(v)$, then every graph $G$ cospectral with $\Gamma$ is
also distance-regular and has the same intersection numbers
as $\Gamma$.
\begin{itemize}
\item[$(i)$]
 $g \ge 2d-1$,
\item[$(ii)$]
$g \ge 2d-2$ and $\Gamma$ is bipartite,
\item[$(iii)$]
$g \ge 2d-2$ and $c_{d-1}c_d < -(c_{d-1}+1)(\lambda_1+\cdots+\lambda_d)$,
\item[$(iv)$]
$\Gamma$ is a generalized Odd graph, i.e., $a_1 = \cdots = a_{d-1} = 0$, $a_d \neq 0$,
\item[$(v)$]
$c_1 = \cdots = c_{d-1} = 1$.
\end{itemize}
\end{theo}

Here we show that the same conclusions can be obtained within the following much more general setting:
\begin{itemize}
\item $G$ has preintersection numbers satisfying certain properties.
\end{itemize}

More precisely, in Theorem \ref{theo-girth} we generalize the cases $(i)$ and $(ii)$ of Theorem \ref{theo-book}. A refinement of $(iii)$ is given in Theorem \ref{theo-girth+}. Moreover, our Theorem \ref{thm:variationoddgirththeorem} provides an alternative formulation of the so-called odd-girth theorem, which is a generalization of $(iv)$. Finally, Theorem \ref{gamma=1} generalizes the case $(v)$.

Our work is motivated by earlier work in this direction, in particular by the odd-girth theorem \cite{vdh11}. This result states that a graph with $d+1$ distinct eigenvalues and odd-girth $2d+1$ is distance-regular. We recall that the odd-girth of a graph is the length of the shortest odd cycle in the graph, and that the odd-girth follows from the spectrum of the graph.

In order to obtain our results we mainly use the theory on so-called almost distance-regular graphs given by Dalf\'{o}, Van Dam, Fiol, Garriga, and Gorissen \cite{ddfgg11}. Another important ingredient of our work is a new inequality (Proposition \ref{theo:excesspdr}) for partially distance-regular graphs that is inspired by Fiol and Garriga's spectral excess theorem \cite{fg97} (for short proofs, see \cite{vd08,fgg10}). The spectral excess theorem  states that, if for every vertex $u$, the number of vertices at (maximum) distance $d$ from $u$ is the same as the so-called spectral excess (which can be expressed in terms of the spectrum), then the graph is distance-regular.

The paper is organized as follows. In Section \ref{prelim}, we give some basic background information. In Section \ref{sec:prepol} we present a few lemmas about properties of the predistance polynomials and preintersection numbers.
These parameters are used in Section \ref{sec-new} to prove
the above-mentioned new characterizations of distance-regularity.

\section{Preliminaries}
\label{prelim}

First, let us first recall
 some basic concepts, notation, and results on which
our study is based. For more background on spectra of
graphs, distance-regular graphs, and their characterizations,
see \cite{b93,bcn89,bh12,cds82,vdkt12,f02,g93}. Throughout this
paper, $G=(V,E)$ denotes a finite, simple, and connected graph
with vertex set $V$, order $n=|V|$, size $e=|E|$, and diameter $D$. The set (`sphere')
of vertices at distance $i=0,\ldots,D$ from a given vertex $u\in V$ is
denoted by $S_i(u)$, and we let $k_i(u)=|S_i(u)|$.
When the numbers $k_i(u)$
do not depend on the vertex $u\in V$, which for example is the case when $G$ is distance-regular, we simply write $k_i$.
For a regular graph, we sometimes abbreviate the valency $k_1$ by $k$.

Recall also that, for every $i=0,\ldots,D$, the distance matrix $\A_i$ has entries $(\A_i)_{uv}=1$ if the distance
between $u$ and $v$, denoted $\dist(u,v)$, is given by $\dist(u,v)=i$, and $(\A_i)_{uv}=0$ otherwise. Thus, $\A_i$ is the adjacency matrix of the distance-$i$ graph $G_i$. In particular, $\A_0=\I$ is the identity matrix, $\A_1=\A$ is the adjacency matrix of $G$. Note that $\A_0+\cdots +\A_D=\J$, the all-ones matrix.

The spectrum of $G$ is defined as the spectrum of $\A$, i.e.,
$\spec G := \spec \A=\{\lambda_0^{m_0},\ldots,\lambda_d^{m_d}\}$,
where the distinct eigenvalues of $\A$ are ordered decreasingly:
$\lambda_0>\cdots >\lambda_d$, and the superscripts
stand for their multiplicities $m_i=m(\lambda_i)$. Note that, since $G$ is connected, $m_0=1$, and if $G$ is regular then $\lambda_0=k$. Throughout the paper, $d$ will denote the number of distinct eigenvalues minus one. It is well-known that the diameter is bounded by this number, i.e., $D \leq d$.
Let $\mu$ be the minimal polynomial of $\A$, that is,
$\mu(x)=\prod_{i=0}^d(x-\lambda_i)$. Then
the {\em Hoffman polynomial} $H$ given by $H(x)=n \mu(x)/\mu(\lambda_0)$ characterizes regularity of $G$ by the condition $H(\A)=\J$ (see Hoffman~\cite{hof63}).

 \section{Orthogonal polynomials and preintersection numbers}
\label{sec:prepol}
Orthogonal polynomials have been useful in the study of distance-regular graphs. Given a graph $G$ with adjacency matrix $\A$, and spectrum $\{\lambda_0^{m_0},\ldots,\lambda_d^{m_d}\}$, we consider the scalar product on the vector space $\Re_d[x]$ of polynomials of degree at most $d$, given by
\begin{equation}
\label{produc}
\langle p, q\rangle_G :=\frac{1}{n}\tr (p(\A)q(\A))=\frac{1}{n} \sum_{i=0}^d m_i
p(\lambda_i) q(\lambda_i).
\end{equation}
Note that the second equality in (\ref{produc}) follows from standard properties of the trace.
Within the vector space of real $n\times n$ matrices, we also use the common
scalar product
\begin{equation*}
\langle \M,\N\rangle:=\frac{1}{n}\som (\M\circ \N) =\frac{1}{n}\tr \M\N^{\top},
\end{equation*}

 where `$\circ$' stands
for the entrywise or Hadamard product, and $\som(\cdot)$ denotes the sum of the
entries of the corresponding matrix. Note that
$\langle p, q\rangle_G=\langle p(\A),q(\A)\rangle$.

Fiol and Garriga \cite{fg97} introduced the {\em predistance polynomials}
$p_0,p_1,\ldots,p_d$ as the unique sequence of orthogonal polynomials (so with $\deg p_i=i$ for $i=0,\ldots,d$) with respect to the scalar product \eqref{produc}
that are normalized as $\|p_i\|_{G}^2=p_i(\lambda_0)$. Like every sequence of orthogonal polynomials, the predistance polynomials satisfy a three-term recurrence
\begin{equation}
\label{3-term-recur}
xp_i=\beta_{i-1}p_{i-1}+\alpha_i p_i+\gamma_{i+1}p_{i+1},\quad i=0,\ldots,d,
\end{equation}
for certain {\em preintersection numbers} $\alpha_i,\beta_i$, and $\gamma_i$, where $\beta_{-1}=\gamma_{d+1}=0$, and $p_{-1}=p_{d+1}=0$. For convenience, we also define the preintersection numbers $\gamma_0=0$ and $\beta_d=0$.
Some basic properties of the predistance polynomials and preintersection numbers are included in the following result (see C\'amara, F\`abrega, Fiol, and Garriga \cite{cffg09}).

\begin{lema}
\label{ortho-pol}
Let $G$ be a graph with average degree $\overline{k}=2e/n$. Then
\begin{itemize}
\item[$(i)$]
$p_0=1$, $p_1=(\lambda_0/\overline{k})x$,
\item[$(ii)$]
$\alpha_i+\beta_i+\gamma_i=\lambda_0$, for $i=0,\ldots,d$,
\item[$(iii)$]
$p_{i-1}(\lambda_0)\beta_{i-1}=p_i(\lambda_0)\gamma_i$, for $i=1,\ldots,d$,
\item[$(iv)$]
$p_0+p_1+\cdots+p_d=H$, with $H$ the Hoffman polynomial,
\item [$(v)$]
The tridiagonal $(d+1)\times(d+1)$ `recurrence matrix' $\R$ given by
$$
\R=\left(\begin{array}
{ccccc}
\alpha_0  & \gamma_1    &            &               &              \\
\beta_0   & \alpha_1    & \gamma_2   &               &              \\
          & \beta_1     & \alpha_2   & \ddots        &              \\
          &             & \ddots     & \ddots        & \gamma_d     \\
          &             &            & \beta_{d-1}   & \alpha_d     \\
\end{array}\right)
$$
has eigenvalues $\lambda_0,\ldots,\lambda_d$.
\end{itemize}
\end{lema}

For vertices $u,v$ at distance $i \le D$, we define $c_i(u,v)=|S_{i-1}(u)\cap S_{1}(v)|$,
$a_i(u,v)=|S_i(u)\cap S_{1}(v)|$, and $b_i(u,v)=|S_{i+1}(u)\cap S_{1}(v)|$.  We say that the intersection number $c_i$ is
{\em well-defined} if the numbers $c_i(u,v)$ are the same for all vertices
$u,v$ at distance $i$. Similarly, we define when $a_i$ and $b_i$ are well-defined and say that $k_i$ is well-defined if $k_i(u)$ is the same for every vertex $u$.
Note that $c_i(u,v)+a_i(u,v)+b_i(u,v)=k_1(v)$. This implies that if the graph is regular and $c_i$ and $a_i$ are well-defined, then so is $b_i$.

When the intersection numbers $c_i,a_i$, and $b_i$ are well-defined for all $i=0,\dots,D$, we say that the graph is distance-regular. In this case $D=d$ and the predistance polynomials become the distance
polynomials, so that $p_i(\A)=\A_i$ and
$p_i(\lambda_0)=k_i$ for $i=0,\ldots, D$. Moreover, the preintersection numbers $\gamma_i, \alpha_i$, and $\beta_i$  become the usual
intersection numbers $c_i, a_i,$ and $b_i$, respectively. Analogous to \eqref{3-term-recur}, we then get the recurrence
$$
\A\A_i = b_{i-1}\A_{i-1} +a_i \A_i +c_{i+1}\A_{i+1},
\qquad i =0,\ldots,D.
$$
For an arbitrary graph, we also consider the following averages: $\overline{c}_{i}$ is the average
of the numbers $c_{i}(u,v)$ over all (ordered) pairs of vertices $u,v$ at distance $i$, and similarly we define $\overline{a}_i$ and $\overline{b}_i$.
Also, let $\overline{c_i^2}$ be the average of $c_i(u,v)^2$ over all (ordered) pairs of vertices $u,v$ at distance $i$ and similarly we define $\overline{a_i^2}$ and $\overline{b_i^2}$.
Finally, $\overline{k}_i=\frac1n \sum_{u\in V}k_{i}(u)$.

\begin{lema}
\label{lem:averages} For $i=0,\dots,D$, the following properties hold:
\begin{itemize}
\item[$(i)$]
$c_i(u,v)=(\A\A_{i-1})_{vu}$, $a_i(u,v)=(\A\A_{i})_{vu}$, and $b_i(u,v)=(\A\A_{i+1})_{vu}$,
\item[$(ii)$]$\overline{k}_i={\|\A_i\|^2}$,
\item[$(iii)$]$\overline{c}_{i}=\frac{\langle \A\A_{i-1}, \A_i\rangle}{\|\A_i\|^2},\overline{a}_{i}=\frac{\langle \A\A_{i}, \A_i\rangle}{\|\A_i\|^2}$, and $\overline{b}_{i}=\frac{\langle \A\A_{i+1}, \A_i\rangle}{\|\A_i\|^2}$,
\item[$(iv)$]$\overline{c_i^2}=\frac{\langle \A\A_{i-1} \circ \A\A_{i-1}, \A_i\rangle}{\|\A_i\|^2},\overline{a_i^2}=\frac{\langle \A\A_{i}\circ \A\A_{i}, \A_i\rangle}{\|\A_i\|^2}$, and $\overline{b_i^2}=\frac{\langle \A\A_{i+1}\circ \A\A_{i+1}, \A_i\rangle}{\|\A_i\|^2}$.
\end{itemize}
\end{lema}

\begin{proof} $(i)$ $c_i(u,v)=|S_{i-1}(u)\cap S_{1}(v)|=\sum_{w \in V}\A_{vw}(\A_{i-1})_{wu}=(\A\A_{i-1})_{vu}$. The other expressions follow in a similar manner.

$(ii)$ $\overline{k}_i=\frac1n \sum_{u\in V}k_{i}(u) =\frac1n \sum_{u,v \in V} (\A_i)_{uv}=\frac1n \sum_{u,v \in V} (\A_i)_{uv}(\A_i)_{uv}=\langle \A_i,\A_i \rangle={\|\A_i\|^2}$.

$(iii)$ From using $(i)$ and $(ii)$, it follows that
$$\overline{c}_{i}= \frac{1}{n\overline{k}_i}\sum_{u\in
V}\sum_{v\in
S_i(u)}c_i(u,v)= \frac{1}{n\overline{k}_i}\sum_{u,v\in
V}(\A_i)_{vu}(\A\A_{i-1})_{vu}=\frac{\langle \A\A_{i-1}, \A_i\rangle}{\|\A_i\|^2}.$$
The other expressions follow in a similar manner.

$(iv)$ This follows in a similar manner as $(iii)$.
\end{proof}

We remark that the intersection numbers $c_i(u,v), a_i(u,v),$ and $b_i(u,v)$ are not necessarily symmetric in $u$ and $v$, and hence neither are the products $\A\A_{i-1}$, $\A\A_{i}$, and $\A\A_{i+1}$ necessarily symmetric matrices.
We also note the resemblance of the expressions in Lemma \ref{lem:averages}$(iii)$ and
$$
\gamma_i=\frac{\langle
xp_{i-1},p_{i}\rangle_G}{\|p_{i}\|_G^2}, \alpha_i=\frac{\langle
xp_{i},p_{i}\rangle_G}{\|p_{i}\|_G^2}, \text{ and } \beta_i=\frac{\langle
xp_{i+1},p_{i}\rangle_G}{\|p_{i}\|_G^2},
$$
for $i=0,\dots,d$, which follow from \eqref{3-term-recur}. The expressions in Lemma \ref{lem:averages}$(iv)$ lead to the following resembling results (where we define $\overline{k}_{-1}, \overline{b_{-1}^2}, \overline{k}_{D+1},$ and $\overline{c_{D+1}^2}$ to be $0$):

\begin{lema}
\label{lem:squares} The following properties hold:
\begin{itemize}
\item[$(i)$]
$\| \A\A_{i} \|^2 =\overline{k}_{i-1}\overline{b_{i-1}^2}+\overline{k}_{i}\overline{a_{i}^2}+\overline{k}_{i+1}\overline{c_{i+1}^2}$ for $i=0,\dots,D$,
\item[$(ii)$]
$\| xp_i \|_G^2 =p_{i-1}(\lambda_0)\beta_{i-1}^2+p_{i}(\lambda_0)\alpha_{i}^2+p_{i+1}(\lambda_0)\gamma_{i+1}^2$ for $i=0,\dots,d$.
\end{itemize}
\end{lema}

\begin{proof}
  $(i)$ We first note that $\A\A_{i}=\A\A_{i} \circ (\A_{i-1}+\A_i+\A_{i+1})$ because $(\A\A_i)_{uv}=0$ if the distance between $u$ and $v$ is not $i-1,i,$ or $i+1$. Therefore
\begin{align*}
\| \A\A_{i} \|^2 &=\frac1n \som(\A\A_{i} \circ \A\A_{i})=\frac1n \som(\A\A_{i} \circ \A\A_{i} \circ (\A_{i-1}+\A_i+\A_{i+1}))\\
&=\langle \A\A_{i} \circ \A\A_{i}, \A_{i-1} \rangle+\langle \A\A_{i} \circ \A\A_{i}, \A_{i} \rangle+\langle \A\A_{i} \circ \A\A_{i}, \A_{i+1} \rangle.
\end{align*}
 Using Lemma \ref{lem:averages}$(ii)$ and $(iv)$, the result now follows.

$(ii)$ Because the predistance polynomials are orthogonal polynomials, we obtain that
\begin{align*}
\| xp_i \|_G^2 &=\| \beta_{i-1}p_{i-1}+\alpha_i p_i+\gamma_{i+1}p_{i+1} \|_G^2
=\beta_{i-1}^2 \| p_{i-1}\|_G^2 +\alpha_{i}^2 \| p_{i} \|_G^2 +\gamma_{i+1}^2 \| p_{i+1} \|_G^2 \\ &=p_{i-1}(\lambda_0)\beta_{i-1}^2+p_{i}(\lambda_0)\alpha_{i}^2+p_{i+1}(\lambda_0)\gamma_{i+1}^2. \qedhere
\end{align*}
\end{proof}

Furthermore, we need the following properties of the predistance polynomials and preintersection numbers.

\begin{lema}
\label{two-terms}
For $i=0,\ldots,d$, the two highest terms of the predistance polynomial $p_i$ are given by
$$
\textstyle
p_i(x)=\frac{1}{\gamma_1\cdots\gamma_i}[x^i - (\alpha_1+\cdots+\alpha_{i-1})x^{i-1}+\cdots].
$$
\end{lema}
\begin{proof}
Use induction by using the three-term recurrence \eqref{3-term-recur} and initial value $p_0=1$.
\end{proof}
The following result is a straightforward consequence of Lemmas \ref{ortho-pol}$(i)$ and \ref{two-terms} and the fact that $G$ is regular if and only
if $\lambda_0=\overline{k}$ (see e.g. Brouwer and Haemers \cite{bcn89}):

\begin{lema}\label{lemma_regulargraph} Let $G$ be a graph. Then the following properties are equivalent:
$(i)$ $G$ is regular; $(ii)$ $p_1=x$; and $(iii)$ $\gamma_1=1$.
\end{lema}

It is clear that the intersection numbers $a_i$, $b_i$, and $c_i$ of a distance-regular graph are nonnegative integers with precise combinatorial meanings. In contrast, this does not hold for the corresponding preintersection numbers $\alpha_i$, $\beta_i$, and $\gamma_i$, which in general are not even integers. Nevertheless, they do share some properties, as shown in Lemma \ref{ortho-pol} and in the following result.

\begin{lema}
\label{properties-preintersectionnumbers}
Let $G$ be a graph with distinct eigenvalues $\lambda_0>\cdots >\lambda_d$, and preintersection numbers $\alpha_i$, $\beta_i$, and $\gamma_i$. Then
\begin{itemize}
\item[$(i)$]
$\gamma_{i}>0$ for
$i=1,\ldots,d$, and $\beta_{i}>0$ for $i=0,\ldots,d-1$,
\item[$(ii)$]
$\sum_{i=0}^d \alpha_i=\sum_{i=0}^d \lambda_i$.
\end{itemize}
\end{lema}
\begin{proof}
$(i)$ First note that $p_i(\lambda_0)=\|p_i\|_{G}^2>0$ for every $i=0,\ldots,d$. Thus, by Lemma \ref{ortho-pol}$(iii)$, we only need to prove the condition on the $\gamma_i$'s. Moreover, by the interlacing property of orthogonal
polynomials, we know that all the zeros of $p_i$ lie between $\lambda_d$ and
$\lambda_0$. Consequently, the leading coefficient $\omega_i$ of $p_i$ must
be positive, as $\lim_{x\rightarrow \infty}p_i(x) =
\infty$. Thus, the conclusion is obtained by Lemma \ref{two-terms} since
$\omega_i=(\gamma_1\cdots\gamma_i)^{-1}$ for $i=1,\ldots,d$.
To prove $(ii)$, just use Lemma \ref{ortho-pol}$(v)$ and consider the trace of the recurrence matrix $\R$.
\end{proof}

In contrast to $\gamma_{i}>0$ and $\beta_{i}>0$, there are graphs such
that $\lambda_0+\cdots+\lambda_d<0$ and, hence, by Lemma
\ref{properties-preintersectionnumbers}$(ii)$,  some of their
preintersection numbers $\alpha_i$ must be negative. For instance, this is the case for
the cubic graph on twelve vertices listed as no. $3.83$ in \cite{cds82}.

Finally, by the same inductive argument used by Van Dam and Haemers \cite{vdh11}, we have that the {\em odd-girth} of a graph (that is, the length of its shortest odd cycle) can be determined from the preintersection numbers as follows.

\begin{lema}
\label{oddgirthlemma}
A non-bipartite graph has odd-girth $2m+1$ if and only if $\alpha_0=\cdots=\alpha_{m-1}=0$ and $\alpha_m \neq 0$. A graph is bipartite if and only if $\alpha_0=\cdots=\alpha_{d}=0$.
\end{lema}

Note that in general, the girth is not determined by the spectrum, but for regular graphs it is. In Corollary \ref{cor:girth} we will make this explicit in terms of the preintersection numbers.

\section{New quasi-spectral characterizations of distance-regular graphs}
\label{sec-new}
This section contains the main results of our work. We will give sufficient conditions for a graph to be distance-regular, without requiring the graph to be cospectral with a distance-regular graph.
We begin with an alternative formulation of the so-called odd-girth theorem \cite{vdh11}.

\subsection{The odd-girth theorem revisited}
\label{odd-girth}

Theorem \ref{theo-book}$(iv)$ was generalized by Van Dam and Haemers \cite{vdh11} as the odd-girth theorem, which
states that a graph $G$ with $d+1$ distinct eigenvalues and odd-girth $2d+1$ is distance-regular.
By Lemma \ref{oddgirthlemma}, the condition on the odd-girth of $G$ is equivalent to $\alpha_1=\cdots=\alpha_{d-1}=0$, $\alpha_d\neq 0$, which corresponds to the condition $a_1=\cdots=a_{d-1}=0$, $a_d\neq 0$ of Theorem \ref{theo-book}$(iv)$. Note that Lee and Weng \cite{lw11} and Van Dam and Fiol \cite{vdf12} showed that the odd-girth theorem is not restricted to regular graphs.

Before presenting an alternative formulation of the odd-girth theorem, recall that a {\em generalized Odd graph} is a distance-regular graph with diameter $D$ and odd-girth $2D+1$. A well-known example is the {\em Odd graph} $O_{D+1}$, whose vertices represent the $D$-element subsets of a $(2D+1)$-element set, where two vertices are adjacent if and only if their corresponding subsets are disjoint, see Biggs \cite{b93}.

\begin{theo}
\label{thm:variationoddgirththeorem}
Let $G$ be a non-bipartite graph with $d+1$ distinct eigenvalues.
\begin{itemize}
\item[$(i)$]
 If $\alpha_i\ge 0$ for $i=0,\ldots,d-1$, then
\begin{equation*}
\gamma_d\ge -(\lambda_1+\cdots+\lambda_d),
\end{equation*}
with equality if and only if $G$ is a (distance-regular) generalized Odd graph.
\item[$(ii)$]
If $G$ has odd-girth at least $2d-1$ and $\gamma_d=-(\lambda_1+\cdots+\lambda_d)$, then $G$ is a (distance-regular) generalized Odd graph.
\end{itemize}
\end{theo}
\begin{proof}
We will use that $\alpha_{0}+\cdots+\alpha_d=\lambda_0+\cdots+\lambda_d$ (by Lemma \ref{properties-preintersectionnumbers}$(ii)$) and $\alpha_d+\gamma_d=\lambda_0$ (by Lemma \ref{ortho-pol}$(ii)$ and recalling that $\beta_d=0$). To show $(i)$, observe that the hypothesis now implies that
$$
\gamma_d=\lambda_0-\alpha_d=-(\lambda_1+\cdots + \lambda_d)+(\alpha_0+\cdots+\alpha_{d-1})\ge -(\lambda_1+\cdots + \lambda_d),
$$
with equality if and only if $\alpha_0=\cdots =\alpha_{d-1}=0$. Because $G$ is not bipartite, this is equivalent to the odd-girth of $G$ being $2d+1$, and so $(i)$ follows from the odd-girth theorem.

To show $(ii)$, note that by Lemma \ref{oddgirthlemma} we have that $\alpha_0=\cdots=\alpha_{d-2}=0$, and hence $\alpha_{d-1}+\alpha_d=\lambda_0+\cdots+\lambda_d$. This implies that
\begin{equation}
\label{gamma-alpha}
\gamma_d-\alpha_{d-1}=-(\lambda_1+\cdots+\lambda_d),
\end{equation}
and so, by the assumption, $\alpha_{d-1}= 0$. Hence $G$ has odd-girth $2d+1$, and $(ii)$ follows, again by the odd-girth theorem.
\end{proof}

We will make further use of \eqref{gamma-alpha} in subsequent sections on graphs with large girth. There (Theorem \ref{theo-girth+}) we will also present a variation of Theorem \ref{thm:variationoddgirththeorem}$(i)$.

Of course, one of the cases (but certainly not the only one) where the hypothesis that $\alpha_i \geq 0$ for $i=0,\ldots,d-1$ holds, is when $G$ is cospectral with a distance-regular graph. However, as we mentioned above, the hypothesis is not satisfied in general.

In contrast to the above, if $G$ is bipartite, then $\gamma_d=-(\lambda_1+\cdots+\lambda_d)$, but in general we cannot conclude that $G$ is distance-regular. A counterexample is the Hoffman graph \cite{hof63}, which is cospectral with the distance-regular $4$-cube
$Q_4$, and hence it is bipartite with $d=4$ (because $\alpha_0=\cdots=\alpha_4=0$). The Hoffman graph is not distance-regular however.

\subsection{Distance-regularity from large girth}
\label{dr-lgirth}
From now on, we will use basic results on partially distance-regular graphs, whose definition is as follows.
A graph $G$ with diameter $D$ is called {\em $m$-partially distance-regular}, for some $m=0,\dots, D$, if its predistance polynomials satisfy $p_i(\A)=\A_i$ for every $i\le m$. In particular,  every $m$-partially distance-regular with $m\ge 1$ must be regular. This is because $p_1(\A)=\A$ and hence $p_1=x$ is equivalent to $G$ being regular, by Lemma \ref{lemma_regulargraph}. As an alternative characterization, we have that $G$ is $m$-partially distance-regular when the intersection numbers $c_i$ ($i \leq m$), $a_i$ ($i \leq m-1$), $b_i$ ($i \leq m-1$) are well-defined.
In this case, these intersection numbers are equal to the corresponding preintersection numbers $\gamma_i$ ($i \leq m$), $\alpha_i$ ($i \leq m-1$), $\beta_i$ ($i \leq m-1$), and also $k_i$ is well-defined and equal to $p_i(\lambda_0)$ for $i \le m$.  We refer to Dalf\'{o}, Van Dam,  Fiol, Garriga, and Gorissen \cite{ddfgg11} for more background.
Our second main result uses the two following results from \cite{ddfgg11}.
\begin{lema}
\label{lema-girth}
Let $G$ be a regular graph with girth $g$.
Then $G$ is $m$-partially distance-regular with $m = \lfloor (g-1)/2\rfloor$ and intersection numbers $a_i=0$ for $i=0,\dots,m-1$ and $c_i=1$ for $i=1,\dots,m$.
\end{lema}

\begin{propo}
\label{propo-pdrg=drg}
Let $G$ be a graph with $d+1$ distinct eigenvalues.
\begin{itemize}
\item[$(i)$]
If $G$ is $(d-1)$-partially distance-regular, then $G$ is distance-regular,
\item[$(ii)$]
If $G$ is bipartite and $(d-2)$-partially distance-regular, then $G$ is distance-regular.
\end{itemize}
\end{propo}

The cases $(i)$ and $(ii)$ of Theorem \ref{theo-book} can now easily be generalized as
follows.

\begin{theo}
\label{theo-girth}
A regular graph $G$ with girth $g$
is distance-regular if either one of the following conditions holds:
\begin{itemize}
\item[$(i)$]
 $g \ge 2d-1$,
\item[$(ii)$]
$g \ge 2d-2$ and $G$ is bipartite.
\end{itemize}
\end{theo}

\begin{proof}
$(i)$ If $g \ge 2d-1$, then $G$ is $(d-1)$-partially distance-regular by Lemma \ref{lema-girth}, and the result follows from Proposition \ref{propo-pdrg=drg}$(i)$. The proof of $(ii)$ is similar by using Proposition \ref{propo-pdrg=drg}$(ii)$.
\end{proof}

We recall that the condition of being bipartite follows from the spectrum, and also the girth of a regular graph is determined by the
spectrum. Thus, the assumptions in Theorem \ref{theo-girth} only depend on the spectrum of $G$.

\subsection{Distance-regularity from the (pre)intersection numbers}
\label{dr-preintnum}
In this section we will give several characterizations of distance-regularity that involve the preintersection numbers.
With this aim in mind, we start with deriving some properties of the preintersection numbers of $(m-1)$-partially distance-regular graphs.

\begin{lema}
\label{lema:basic00}
Let $G$ be a regular graph
and let $m\le D$ be a positive integer. Suppose that $G$ is $(m-1)$-partially distance-regular. Then the following properties hold:
\begin{itemize}
\item[$(i)$]
 $\alpha_{m-1}=\overline{a}_{m-1}$
 and $\beta_{m-1}=\overline{b}_{m-1}=\frac{\overline{k}_m\overline{c}_m}{k_{m-1}}$,
\item[$(ii)$]
$k_{m-1}\alpha_{m-1}^2+ p_m(\lambda_0)\gamma_m^2 = k_{m-1}\overline{a_{m-1}^2}+\overline{k}_m \overline{c_m^2}$,
\item[$(iii)$]
$p_m(\lambda_0)\gamma_{m}^2\ge \overline{k}_m\overline{c_m^2}$, with equality if and only if $a_{m-1}$ is well-defined,
\item[$(iv)$]
if $a_{m-1}$ is well defined, then $\gamma_{m}=\frac{\overline{c_m^2}}{\overline{c}_m}$.
\end{itemize}
\end{lema}

\begin{proof}

Since $G$ is $(m-1)$-partially distance-regular, among others the intersection numbers $c_{m-1}$ and $b_{m-2}$ are well-defined and equal to $\gamma_{m-1}$ and $\beta_{m-2}$, respectively, and also $k_{m-2}$ and $k_{m-1}$ are well-defined and equal to $p_{m-2}(\lambda_0)$ and  $p_{m-1}(\lambda_0)$, respectively. Moreover, $p_{m-1}(\A) = \A_{m-1}$.

$(i)$ By using Lemma \ref{lem:averages}$(iii)$, it now follows that
$$\alpha_{m-1} =\frac{\langle x p_{m-1},p_{m-1} \rangle_G}{\|p_{m-1}\|^2_G}=\frac{\langle \A\A_{m-1}, \A_{m-1} \rangle}{\|\A_{m-1}\|^2}=\overline{a}_{m-1}.$$

From this it follows that $\beta_{m-1}=k-\gamma_{m-1}-\alpha_{m-1}=k-c_{m-1}-\overline{a}_{m-1}=\overline{b}_{m-1}$, where we also used Lemma
\ref{ortho-pol}$(ii)$ and that $G$ is regular with valency $k=\lambda_0$. Moreover, by Lemma \ref{lem:averages}$(ii)$ and $(iii)$,
$k_{m-1}\overline{b}_{m-1} = \langle \A\A_m,\A_{m-1} \rangle= \langle \A_m,\A\A_{m-1} \rangle=\overline{k}_m\overline{c}_{m}$, hence $\overline{b}_{m-1}=\overline{k}_m\overline{c}_m/k_{m-1}$.

$(ii)$ This follows from Lemma \ref{lem:squares} and working out the equation $\| \A\A_{m-1} \|^2=\| xp_{m-1} \|_G^2$, while using that $\overline{k}_{m-2}=p_{m-2}(\lambda_0)$, $\overline{b_{m-2}^2}=\beta_{m-2}^2$ (because $b_{m-2}$ is well-defined), and $\overline{k}_{m-1}=p_{m-1}(\lambda_0)=k_{m-1}$.

$(iii)$ By using $(i)$, it follows that
$$
\overline{a_{m-1}^2} \ge (\overline{a}_{m-1})^2= \alpha_{m-1}^2,
$$
with equality if and only if $a_{m-1}$ is well-defined (and $a_{m-1}=\alpha_{m-1}$).
The statement now follows from combining this with $(ii)$.

$(iv)$ Using $(i)$ and Lemma \ref{ortho-pol}$(iii)$, we obtain that $p_m(\lambda_0)\gamma_m =
p_{m-1}(\lambda_0)\beta_{m-1}=k_{m-1}\beta_{m-1}=\overline{k}_m
\overline{c}_m$. Thus, from this and $(iii)$,
\begin{align*}
\overline{k}_m \overline{c_m^2}= p_m(\lambda_0) \gamma_m^2 = \overline{k}_m  \overline{c}_m \gamma_m,
\end{align*}
whence the result follows.
\end{proof}

The following observation is the key to many of our results. It is motivated by the spectral excess theorem, and will be used to prove Proposition \ref{theo:basic}.

\begin{propo}
\label{theo:excesspdr}
Let $G$ be a regular graph with diameter $D$,
and let $m\le D$ be a positive integer. If $G$ is $(m-1)$-partially distance-regular, then $\overline{k}_m \geq p_m(\lambda_0)$ with equality if and only if $G$ is $m$-partially distance-regular.
\end{propo}

\begin{proof} Assume that $G$ is $(m-1)$-partially distance-regular. Then $p_i(\A)=\A_i$ and hence $\langle p_m(\A), \A_i\rangle=\langle p_m, p_i\rangle_G=0$ for $i<m$.
Moreover, $(p_m(\A))_{uv}=0$ for every pair of vertices $u,v$ at distance $i>m$ and hence $\langle p_m(\A), \A_i\rangle=0$ also for $i>m$.
This implies that
$$\langle p_m(\A),\A_m\rangle = \langle p_m(\A),\J\rangle=\langle
p_m,H\rangle_G=\langle
p_m,p_0+\cdots+p_d\rangle_G=\langle
p_m,p_m\rangle_G=p_m(\lambda_0),$$
where we used Lemma \ref{ortho-pol}$(iv)$ and that $H(\A)=\J$. Then, by the Cauchy-Schwarz
inequality, $p^2_m(\lambda_0) \le \|p_m(\A)\|^2\|\A_m\|^2 = p_m(\lambda_0)\overline{k}_m$, and hence $\overline{k}_m \geq p_m(\lambda_0)$.
Furthermore, in the case of equality, $p_m(\A)=\alpha \A_m$ for some $\alpha\in \Re$, and by taking norms we get that
$\alpha=1$ since $p_m(\lambda_0)>0$.
\end{proof}

The following result generalizes some results by Van Dam and Haemers \cite{vdh02}, and Van Dam, Haemers, Koolen, and Spence \cite{vdhks06}. Several results in this section will be derived from it.

\begin{propo}
\label{theo:basic}
Let $G$ be a regular graph
and let $m\le D$ be a positive integer. Suppose that $G$ is $(m-1)$-partially distance-regular and any of the following conditions holds:
\begin{itemize}
\item[$(i)$]
$\overline{c}_m\ge \gamma_m$,
\item[$(ii)$]
$c_{m-1}\ge \gamma_m$,
\item[$(iii)$]
$k_{m-1}(\overline{a_{m-1}^2}-\alpha_{m-1}^2)+\overline{k}_m(\overline{c_m^2}-\gamma_m^2)\ge 0$,
\item[$(iv)$]
$\overline{c_m^2}\ge \gamma_m^2$,
\item[$(v)$]
$a_{m-1}$ is well-defined, and $c_m(u,v)\le \gamma_m$ for every pair of vertices $u,v$ at distance $m$.
\end{itemize}
Then $G$ is $m$-partially
distance-regular.
\end{propo}

\begin{proof}
$(i)$ By Lemma \ref{lema:basic00}$(i)$, Lemma \ref{ortho-pol}$(iii)$, and the hypothesis,
\begin{equation*}
\label{ineq-1}
\overline{k}_m= \frac{1}{\overline{c}_m}k_{m-1}\beta_{m-1}=\frac{1}{\overline{c}_m}p_{m-1}(\lambda_0)\beta_{m-1}=\frac{1}{\overline{c}_m} p_m(\lambda_0)\gamma_m\le p_m(\lambda_0).
\end{equation*}
Now the conclusion follows from Proposition \ref{theo:excesspdr}.

$(ii)$ If $u$ and $v$ are vertices at distance $m$, and $u'$ is adjacent to $u$ and on a shortest path between $u$ and $v$, then $S_{m-1}(u) \cap S(v) \supset S_{m-2}(u') \cap S(v)$, which shows that $c_m(u,v)\ge c_{m-1}(u',v)= c_{m-1}\geq \gamma_m$. Now the result follows from $(i)$.

\item[$(iii)$]
By Lemma \ref{lema:basic00}$(ii)$ and the hypothesis, we have
$$
p_m(\lambda_0)=\frac{k_{m-1}(\overline{a_{m-1}^2}-\alpha_{m-1}^2)+\overline{k}_m\overline{c_m^2}}{\gamma_m^2}\ge
\frac{\overline{k}_m \gamma_m^2}{\gamma_m^2}=\overline{k}_m,
$$
and the result follows from Proposition \ref{theo:excesspdr}.

$(iv)$ From Lemma \ref{lema:basic00}$(iii)$ and the hypothesis, we have that $p_m(\lambda_0)\ge \overline{k}_m$,
and the result follows again from Proposition \ref{theo:excesspdr}.

$(v)$ Because $a_{m-1}$ is well-defined, also $b_{m-1}$ is well-defined, and $\overline{c_m^2}= \gamma_m\overline{c}_m$ by Lemma \ref{lema:basic00}$(iv)$. But if $c_m(u,v) \le \gamma_m$ for every pair of vertices $u,v$ at distance $m$, then this easily shows that the intersection
number $c_m$ is also well-defined, which proves the result.
\end{proof}

Since $\overline{c_m^2}\ge (\overline{c}_m)^2$, the result with condition $(i)$ is a consequence of the result involving condition $(iv)$.
Also, observe that, because of Lemma \ref{lema:basic00}$(i)$, the proof of Proposition \ref{theo:basic}$(v)$ also works if we change
the hypothesis `$a_{m-1}$ is well-defined' to either
`$a_{m-1}(u,v)\le \alpha_{m-1}$ for every  $u,v$ at distance $m-1$' or
`$a_{m-1}(u,v)\ge \alpha_{m-1}$ for every  $u,v$ at distance $m-1$'.
The result also holds if we require that
`$c_{m}(u,v)\ge \gamma_{m}$ for every  $u,v$ at distance $m$', in which case we do not need the above hypotheses on $a_{m-1}$
 since then $\overline{c}_m\ge \gamma_m$ and the result follows from Proposition \ref{theo:basic}$(i)$.

As a consequence of Proposition \ref{theo:basic}$(i)$, and since every regular graph is clearly $1$-partially distance-regular with $c_1=\gamma_1=1$ by Lemma \ref{lemma_regulargraph}, we have the following result.

\begin{propo}
\label{c=gamma}
\begin{itemize}
\item[$(i)$]
Every regular graph $G$ with $D \ge d-1$ and preintersection numbers
satisfying $\overline{c}_i\ge \gamma_i$ for $i=2,\ldots,d-1$, is distance-regular,
\item[$(ii)$]
Every regular bipartite graph $G$ with $D \ge d-2$ and preintersection
numbers satisfying $\overline{c}_i\ge \gamma_i$ for $i=2,\ldots,d-2$, is distance-regular.
\end{itemize}
\end{propo}

\begin{proof}
$(i)$ Apply Proposition \ref{theo:basic}$(i)$ recursively to show that $G$ is $(d-1)$-partially distance-regular and then use Proposition \ref{propo-pdrg=drg}$(i)$. The proof of $(ii)$ is similar.
\end{proof}

From Proposition \ref{c=gamma}$(i)$, it clearly follows that if $G$
has the parameters $c_i$ well-defined and equal to $\gamma_i$
for $i=1,\ldots,d-1$, then $G$ is distance-regular. Note that it
is not enough to assume only that the $c_i$'s are well-defined.
To illustrate this, we give an example of a
non-distance-regular graph with well-defined $k_i$ and $c_i$.
Consider the strong product $G$ of the cube $Q_3$ with the complete graph $K_2$, shown in Figure \ref{fig3}. This graph
is $7$-regular with spectrum $\spec G=\{7^{1},3^{3},-1^{11},-5^{1}\}$,
it has diameter $D=d=3$, and well-defined intersection numbers $c_1=1$, $c_2=4$, and
$c_3=6$. However, it is not a distance-regular graph. We note that
$G$ has preintersection numbers $\gamma_1=1$, $\gamma_2\approx4.571$
and $\gamma_3\approx4.816$. Even more so, it has well-defined $k_1=7$,
$k_2=6$, and $k_3=2$ (which is easily seen because $G$ is vertex-transitive). In fact, only $a_1$ and $b_1$ are not
well-defined.

\begin{figure}[t]
\begin{center}
\includegraphics[scale=0.7]{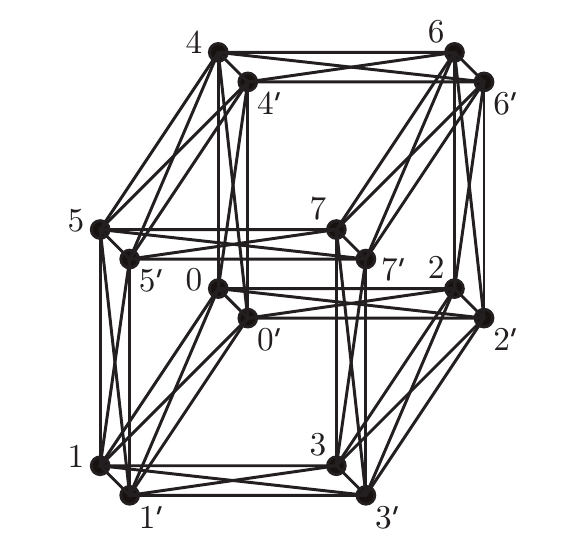}
\caption{The strong product of $Q_3$ by $K_2$.}
\label{fig3}
\end{center}
\end{figure}

Similarly, if one takes the Kronecker product of the adjacency matrix of a bipartite
distance-regular graph with even diameter $D$ with the all-one
$2\times 2$ matrix $\J$, then the result is the adjacency matrix of a regular graph with diameter $D=d$
and with well-defined $k_i$ and $a_i$, but it is not
distance-regular, since $c_2$ and $b_2$ are not well-defined.

These examples show that the combinatorial properties are not sufficient and some extra spectral information is required. This is in line with earlier results in the literature, where cospectrality with a
distance-regular graph, or `feasible spectrum' for a distance-regular graph, is required (see, for example, Haemers \cite{h96} or Van Dam and Haemers \cite{vdh02}).

Another consequence of Proposition \ref{theo:basic} is the following result. It corresponds to the result of Van Dam and Haemers \cite{vdh02} stated in Theorem \ref{theo-book}$(v)$, and its bipartite counterpart. Recall that the preintersection numbers are determined by the spectrum, and that regularity of a graph is characterized by the condition that $\gamma_1=1$.

\begin{theo}
\label{gamma=1}
Let $G$ be a graph with $d+1$ distinct eigenvalues.
\begin{itemize}
\item[$(i)$]
If $d\geq 2$ and $G$ has preintersection numbers $\gamma_1=\cdots=\gamma_{d-1}=1$, then it is distance-regular,
\item[$(ii)$]
If $d\geq 3$ and $G$ is bipartite and has preintersection numbers $\gamma_1=\cdots=\gamma_{d-2}=1$, then it
is distance-regular.
\end{itemize}
\end{theo}

\begin{proof}
$(i)$ If $D \le d-1$, then apply Proposition \ref{theo:basic}$(i)$ or $(ii)$ recursively (using that $\overline{c}_m \geq 1$ and $c_{m-1} \geq 1$) to derive that $G$ is $D$-partially distance-regular, that is, that $G$ is distance-regular. If $D=d$, then it follows similarly that $G$ is $(d-1)$-partially distance-regular, and then it follows from Proposition \ref{propo-pdrg=drg}$(i)$ that $G$ is distance-regular. The proof of $(ii)$ is similar.
\end{proof}

Moreover, Proposition \ref{theo:basic}$(ii)$ also yields the following slight improvement of Proposition \ref{propo-pdrg=drg}. Recall that $1$-partial distance-regularity implies regularity.

\begin{propo}
\label{theoc:c(m)=c(m-1)'}
Let $G$ be a graph  with $d+1$ distinct eigenvalues.
\begin{itemize}
\item[$(i)$]
If $d\geq 3$, $G$ is $(d-2)$-partially distance-regular, and $\gamma_{d-1}\le c_{d-2}$, then $G$ is distance-regular,
\item[$(ii)$]
If $d\geq 4$, $G$ is bipartite and $(d-3)$-partially distance-regular, and $\gamma_{d-2}\le c_{d-3}$, then $G$ is distance-regular.
\hfill $\Box$
\end{itemize}
\end{propo}

\subsection{Distance-regularity from large girth}
\label{dr-lgirth-rev}
Our aim here is to give some results concerning graphs with large girth.
First, we need the following characterization of the girth of a regular graph in terms of the preintersection numbers (cf.~Lemma \ref{oddgirthlemma} for a similar characterization for the odd-girth).

\begin{corollari}\label{cor:girth} \begin{itemize}
\item[$(i)$]
A regular graph has girth $2m+1$ if and only if $\alpha_0=\cdots=\alpha_{m-1}=0$, $\alpha_m \neq 0$, and $\gamma_1=\cdots=\gamma_m=1$,
\item[$(ii)$]
A regular graph has girth $2m$ if and only if $\alpha_0=\cdots=\alpha_{m-1}=0$,  $\gamma_1=\cdots=\gamma_{m-1}=1$, and $\gamma_m > 1$.
\end{itemize}
\end{corollari}

\begin{proof} $(i)$ If the graph is regular with girth $2m+1$, then it is $m$-partially distance-regular with $\alpha_0=\cdots=\alpha_{m-1}=0$ and $\gamma_1=\cdots=\gamma_m=1$ by Lemma \ref{lema-girth}, and $\alpha_m \neq 0$ by Lemma \ref{oddgirthlemma}.

   Conversely, if $\alpha_0=\cdots=\alpha_{m-1}=0$, $\alpha_m \neq 0$, and $\gamma_1=\cdots=\gamma_m=1$, then by combining Lemma \ref{oddgirthlemma} and Proposition \ref{theo:basic}$(ii)$ recursively, it follows that the graph is $m$-partially distance-regular with $a_0=\cdots=a_{m-1}=0$ and $c_1=\cdots=c_m=1$, so the girth is at least $2m+1$. Moreover, by Lemma \ref{oddgirthlemma}, the odd-girth is $2m+1$, which shows that the girth is indeed $2m+1$.

$(ii)$ This follows from similar arguments and using $(i)$.
\end{proof}

From  Proposition \ref{theo:basic}$(v)$, we now obtain a refinement of
the result in Theorem \ref{theo-book}$(iii)$.

\begin{theo}
\label{theo-girth+} Let $G$ be a regular graph with $d+1$
distinct eigenvalues $\lambda_0>\cdots >\lambda_d$ and
girth $g\ge 2d-2$. Then
\begin{equation}
\label{value-of-cd}
\gamma_d\ge -(\lambda_1+\cdots+\lambda_d),
\end{equation}
with equality if and only if $G$ is distance-regular and either
bipartite or a generalized Odd graph.
\end{theo}

\begin{proof}
Note that, from the hypothesis on the girth, Lemma
\ref{lema-girth}, and Corollary \ref{cor:girth}, it follows that $G$ is $(d-2)$-partially distance-regular
with $c_i=\gamma_i=1$ and $a_i=\alpha_i=0$ for $i=1,\ldots,d-2$. It also follows that if $u$ and $v$ are vertices at distance $d-1$, then $c_{d-1}(u,v)=(\A^{d-1})_{uv}$.
Moreover, since
$$
\displaystyle
\sum_{i=0}^d p_i(x)=H(x)= \frac{n}{\pi_0}\prod_{i=1}^d (x-\lambda_i)=\frac{n}{\pi_0}[x^d-(\lambda_1+\cdots+\lambda_d)x^{d-1}+\cdots],
$$
where $\pi_0=\prod_{i=1}^d (\lambda_0-\lambda_i)$, the leading
coefficient of $p_d$ is $n/\pi_0$, and hence $n/\pi_0=(\gamma_d
\gamma_{d-1})^{-1}$ by Lemma \ref{two-terms}.

If we now consider two vertices $u,v$ at distance $d-1$, then
from the equation $H(\A)=\J$ we obtain that
$$
\displaystyle
1=\frac{n}{\pi_0}[(\A^d)_{uv}-(\lambda_1+\cdots+\lambda_d)(\A^{d-1})_{uv}],
$$
and hence that
\begin{equation}
\label{bound-A^d}
(\lambda_1+\cdots+\lambda_d)c_{d-1}(u,v)+\gamma_{d-1}\gamma_d=(\A^d)_{uv}\ge 0.
\end{equation}

Now let us assume that $\gamma_d\le -(\lambda_1+\cdots+\lambda_d)$, and aim to prove equality, so that \eqref{value-of-cd} follows and we can immediately characterize the case of equality.
Then, using the fact that $\gamma_d>0$ (by Lemma
\ref{properties-preintersectionnumbers}$(i)$), we have that
\begin{equation}
\label{ineq:odd-bip}
c_{d-1}(u,v)\le \frac{\gamma_{d-1}\gamma_d}{-(\lambda_1+\cdots+\lambda_d)}\le \gamma_{d-1}.
\end{equation}

Consequently, from Proposition \ref{theo:basic}$(v)$, $G$ is
$(d-1)$-partially distance-regular, and by using Proposition
\ref{propo-pdrg=drg}$(i)$, we conclude that $G$ is
distance-regular, and
$c_{d-1}=\gamma_{d-1}$. Now equalities in
\eqref{ineq:odd-bip} hold for all vertices $u,v$ at distance
$d-1$, and hence we have equality in \eqref{value-of-cd}: $\gamma_d =
-(\lambda_1+\cdots+\lambda_d)$. Moreover, this holds if and only if $(\A^d)_{uv}=0$ in \eqref{bound-A^d}, which means that there are no odd cycles of length smaller than $2d+1$, so $a_0=\cdots = a_{d-1}=0$, and $G$ is either bipartite or a generalized Odd graph. Conversely, when $G$ is bipartite and distance-regular, we have that $\gamma_d=c_d=\lambda_0$ (the degree of $G$) and $\lambda_0+\cdots+\lambda_d=0$ (for example by Lemma \ref{properties-preintersectionnumbers}$(ii)$) and so the condition \eqref{value-of-cd} is tight.
Moreover, when $G$ is a generalized Odd graph, with odd-girth $2d+1$, then $\alpha_d=a_d=\lambda_0+\cdots +\lambda_d$ (this is again a consequence of Lemma \ref{properties-preintersectionnumbers}$(ii)$), and equality in \eqref{value-of-cd} follows from $\alpha_d+\gamma_d=\lambda_0$ (Lemma \ref{ortho-pol}$(ii)$).
\end{proof}

Note that, as a consequence of Theorem \ref{theo-girth+}, the assumptions of Theorem \ref{theo-book}$(iii)$ seem to be quite strong.

By using Proposition \ref{theo:basic}$(i)$, we can also obtain some related results under the assumption that $g\ge 2d-2$.
With this aim, let $\overline{a_{d-1}c_{d-1}}$ be the average
of the products $a_{d-1}(u,v)c_{d-1}(u,v)$ over all pairs $(u,v)$ at distance $d-1$. If $g\ge 2d-2$, then this number equals the average $\overline{a}_{d-1}^{(d)}$  of walks of length $d$ between vertices at distance $d-1$. Indeed, the number of walks of length $d$ between $u$ and $v$ equals $(\A^d)_{uv}$, and hence
\begin{align*}
\overline{a}_{d-1}^{(d)} & =\frac{\langle \A^{d},\A_{d-1}\rangle}{\|\A_{d-1}\|^2} =\frac{\langle \A^{d-1},\A\A_{d-1}\rangle}{\|\A_{d-1}\|^2}
   =\frac{1}{n\overline{k}_{d-1}}\som (\A^{d-1}\circ \A\A_{d-1}) \\
  & = \frac{1}{n\overline{k}_{d-1}}\sum_{u,v\in V}(\A^{d-1})_{vu}(\A\A_{d-1})_{vu}
  = \frac{1}{n\overline{k}_{d-1}}\sum_{\dist(u,v)=d-1}c_{d-1}(u,v)a_{d-1}(u,v).
\end{align*}
Here we have used that $(\A^{d-1})_{vu}=0$ when $\dist(u,v)> d-1$ and, since $g\ge2d-2$, also when
$\dist(u,v)=d-2$. Moreover, $(\A\A_{d-1})_{vu}=0$ when $\dist(u,v)< d-2$.

\begin{propo}
\label{theo-girth++}
Let $G$ be a regular graph with $d+1$ distinct eigenvalues
and girth $g\ge 2d-2$.
\begin{itemize}
\item[$(i)$]
If  $\alpha_{d-1}< \gamma_d$, then $\overline{a_{d-1}c_{d-1}}\ge \alpha_{d-1}\gamma_{d-1}$, with equality if and only if $G$  is distance-regular,
\item[$(ii)$]
If  $\alpha_{d-1}> \gamma_d$, then $\overline{a_{d-1}c_{d-1}}\le \alpha_{d-1}\gamma_{d-1}$, with equality if and only if $G$  is distance-regular,
\item[$(iii)$]
If $\alpha_{d-1}=\gamma_d$, then $(\A^d)_{uv}= \alpha_{d-1}\gamma_{d-1}$ for all vertices $u$ and $v$ at distance $d-1$.
\end{itemize}
\end{propo}
\begin{proof}

If $G$ is distance-regular, then equality in $(i)$ and $(ii)$ is clear since $a_{d-1}=\alpha_{d-1}$ and $c_{d-1}=\gamma_{d-1}$.

On the other hand, \eqref{bound-A^d} and \eqref{gamma-alpha} imply that if $u,v$ are two vertices at distance $d-1$, then
\begin{equation}
\label{bound-A^d(2)}
(\alpha_{d-1}-\gamma_d) c_{d-1}(u,v)+\gamma_{d-1}\gamma_d
=(\A^d)_{uv}.
\end{equation}
Thus, by taking averages over all vertices $u,v$ at distance $d-1$, we have that
$$
(\alpha_{d-1}-\gamma_d)\overline{c}_{d-1}+\gamma_{d-1}\gamma_d
=\overline{a_{d-1}c_{d-1}}.
$$
Suppose now that $\alpha_{d-1}<\gamma_d$. To prove $(i)$, let us now assume that $\overline{a_{d-1}c_{d-1}}\le \alpha_{d-1}\gamma_{d-1}$, and aim to prove equality and that $G$ is distance-regular.
First, we obtain that
\begin{equation*}
\label{mean-c(d-1)=gamma(d-1)}
\overline{c}_{d-1}=\frac{\gamma_{d-1}\gamma_d-\overline{a_{d-1}c_{d-1}}}
{\gamma_d-\alpha_{d-1}}
\ge  \frac{\gamma_{d-1}\gamma_d-\alpha_{d-1}\gamma_{d-1}}{\gamma_d-\alpha_{d-1}}
=\gamma_{d-1}.
\end{equation*}
Then, by Proposition \ref{theo:basic}$(i)$, $G$ is $(d-1)$-partially distance-regular, and the result follows from Proposition \ref{propo-pdrg=drg}$(i)$. The proof of $(ii)$ is similar.

Finally, if the hypothesis in $(iii)$ holds, then \eqref{bound-A^d(2)} gives
$(\A^d)_{uv} =\gamma_{d-1}\gamma_d=\gamma_{d-1}\alpha_{d-1}$
for every pair of vertices $u,v$ at distance $d-1$, as claimed.
\end{proof}

Note that in Proposition \ref{theo-girth++}$(iii)$, it remains open whether the graph must be distance-regular or not. In fact, it is not easy to find graphs with girth $g\ge 2d-2$ satisfying $\alpha_{d-1}=\gamma_d$. Such an example is the Perkel graph \cite{p79} (see also \cite[\S~13.3]{bcn89}), which is a distance-regular graph  with $n=57$ vertices, diameter $D=3$, intersection array $\{b_0,b_1,b_2;c_1,c_2,c_3\}=\{6,5,2;1,1,3\}$, and  spectrum $\{6^{1}, ((3+\sqrt{5})/2)^{18}, ((3-\sqrt{5})/2)^{18}, -3^{20}\}$.  Note that
$\alpha_{2}=\gamma_{3}=3$, as required in the case $(iii)$ of the above result.
Moreover, since $\alpha_1=0$ and $\gamma_2=1$, it has girth $g=5=2d-1$,
so it also satisfies the conditions of Theorem \ref{theo-girth}$(i)$, and hence any graph with the same spectrum is distance-regular. In fact, it is known that this graph is determined by the spectrum, see \cite{vdhks06}.

Another---putative---graph suggests that the graphs in this case need not be distance-regular. It is the first relation in a putative $3$-class association scheme on $81$ vertices, the parameters of which occur on top of p.~102 in the list of \cite{vd99} (with the second relation being the Brouwer-Haemers graph). The spectrum is $\{10^1,1^{20},(-\frac12+\frac12 \sqrt{45})^{30},(-\frac12-\frac12 \sqrt{45})^{30}\}$, and it follows that the (relevant) preintersection numbers are $\alpha_1=0$ (so $g\geq 2d-2$), $\gamma_2=\frac{13}{9}$, and $\alpha_2=\gamma_3=\frac{99}{13}$. Thus, if there exists a graph with this spectrum, then it will not be distance-regular. Now if you consider the graph in the association scheme, then for both types of vertices at distance 2 from a fixed vertex (the type depending on $c_2(u,v)$ being 1 or 2), you can count the number of walks of length 3 using the intersection numbers of the scheme, and indeed in both cases this number equals $\alpha_2\gamma_2=11$. However, note that $a_2(u,v)c_2(u,v)$ is either $9$ or $16$, depending on the relation between $u$ and $v$.

\noindent{\large \bf Acknowledgments.} The authors sincerely acknowledge the relevant contributions of Cristina Dalf\'o at an early stage of this work.
They also thank Andries Brouwer and the anonymous referees for their valuable comments and suggestions, which helped to improve this paper.
This research is supported by the
{\em Ministerio de Ciencia e Innovaci\'on}, Spain, and the {\em European Regional
Development Fund} under project MTM2011-28800-C02-01 (M.A.F.), the {\em Catalan Research
Council} under project 2009SGR1387 (M.A.F.), and by the Netherlands Organization of Scientific Research (NWO) (A.A.).

\end{document}